\documentclass[12pt]{article}
\usepackage{amssymb,amsmath,amsthm,bm,setspace,mathrsfs,paralist}
\usepackage{url,color,dsfont,wasysym}
\usepackage[T1]{fontenc}
\usepackage[margin=.9in]{geometry}
\usepackage[font={footnotesize}]{caption}
\usepackage{graphicx,float}
\usepackage{subfigure,graphics}
\usepackage{xcolor,tikz}
\usepackage{tikz}\usetikzlibrary{shapes,arrows,patterns}
\include{rgb.tex}

\DeclareMathOperator{\Var}{Var}
\DeclareMathOperator{\Cov}{Cov}

\newcommand{\Z}{\mathbb{Z}}

\newcommand{\R}{\mathbb{R}}

\newcommand{\be}{\begin{equation}}
\newcommand{\ee}{\end{equation}}

\renewcommand{\P}{\mathrm{P}}
\newcommand{\E}{\mathrm{E}}

\newcommand{\1}{\mathbb{I}}
\renewcommand{\d}{{\rm d}}

\newcommand{\e}{{\rm e}}

\renewcommand{\leq}{\leqslant}
\renewcommand{\ge}{\geqslant}
\renewcommand{\le}{\leqslant}

\author{ Jingyu Huang\\University of Utah
\and Davar Khoshnevisan\\University of Utah}

\title{\bf On the multifractal local behavior of\\parabolic stochastic PDEs\thanks{
	Research supported in part by the NSF grants DMS-1307470 
	and DMS-1608575.}}
\date{August 3, 2017}

\newtheorem{stat}{Statement}[section]
\newtheorem{proposition}[stat]{Proposition}

\newtheorem{theorem}[stat]{Theorem}
\newtheorem{lemma}[stat]{Lemma}
\theoremstyle{definition} 

\newtheorem{remark}[stat]{Remark}

\numberwithin{equation}{section}
\begin{document}
\maketitle
\begin{abstract}
Consider the stochastic heat equation $\dot{u}=\frac12 u''+\sigma(u)\xi$
on $(0\,,\infty)\times\R$ subject to $u(0)\equiv1$,
where $\sigma:\R\to\R$ is a Lipschitz (local) function that does not vanish at
$1$, and $\xi$ denotes
space-time white noise. It is well known that $u$ has continuous 
sample functions \cite{Walsh}; as a result, $\lim_{t\downarrow0}u(t\,,x)= 1$ almost surely
for every $x\in\R$. 

The corresponding fluctuations are also known \cite{KSXZ,LeiNualart,PT}:
For every fixed $x\in\R$,
$t\mapsto u(t\,,x)$ looks locally like a fixed multiple of fractional Brownian motion 
(fBm) with index $1/4$. In particular,
an application of Fubini's theorem implies that, on an $x$-set of full Lebesgue measure,
the short-time behavior of the  peaks of the random function $t\mapsto u(t\,,x)$ 
are governed by the law of the iterated logarithm for fBm, up to possibly
a suitable normalization constant. By contrast, the main result of this
paper claims that, on an $x$-set of full Hausdorff dimension, the short-time peaks of
$t\mapsto u(t\,,x)$ follow a non-iterated logarithm law, and that those
peaks contain a rich multifractal structure a.s. 

Large-time variations of these results were
predicted in the physics literature a number of years ago and proved very recently
in \cite{KKX1,KKX2}. To the best of our knowledge, the short-time results of the present paper
are observed here for the first time.\\

\noindent{\it Keywords:} The stochastic heat equation, multifractals,
	Hausdorff dimension, packing dimension.\\

	\noindent{\it \noindent AMS 2000 subject classification:}
	Primary 60H15; Secondary 35R60, 60K37. 
\end{abstract}

\section{Introduction}
Let $\sigma:\R\to\R$ be a non-random, Lipschitz continuous function,
$\xi:=$ space-time white noise, and consider the unique,
continuous solution
$u=u(t\,,x)$ to the semi-linear stochastic heat equation,
\begin{equation}\label{SHE}\left[\begin{split}
	&\dot{u} = \tfrac12 u'' + \sigma(u)\xi\quad\text{ on $(0\,,\infty)\times\R$},\\
	&\text{subject to}\ u(0)\equiv1.
\end{split}\right.\end{equation}

It is not hard to see that if $\sigma(1)=0$, 
then $u\equiv1$ a.s. Therefore, we assume that
\[
	\sigma(1)\neq0,
\]
to avoid trivialities.
In this case, it has been shown recently by Khoshnevisan,
Swanson, Xiao, and Zhang \cite{KSXZ} that,
for every fixed $x\in\R$ and $t_0>0$, there exists a fractional Brownian motion 
$B=\{B_t\}_{t\ge0}$ of index
$1/4$ such that 
\begin{equation}\label{approx}
	u(t+t_0\,, x)-u(t_0\,, x) \approx \left({2}/{\pi}\right)^{1/4} 
	\sigma(u(t_0\,, x))B_t
	\qquad\text{a.s.\ when $t\approx0$.}
\end{equation}
For closely-related works,
see also Lei and Nualart \cite{LeiNualart} and Posp\'\i\v{s}il and Tribe
\cite{PT}.
The quality of the approximation \eqref{approx}
is good enough that one can deduce from it a
good deal of local information about the sample functions of $t\mapsto
u(t\,,x)$ near $t=0$. For example, it is possible to prove that
the following law of the iterated logarithm holds for all $x\in\R$
and $t_0>0$:
\begin{equation}\label{LIL}
	\limsup_{\varepsilon\downarrow0} \frac{
	u(t_0+\varepsilon\,,x)-u(t_0\,,x)}{%
	\varepsilon^{1/4}\sqrt{\log\log(1/\varepsilon)}}=
	(8/\pi)^{1/4}\sigma(u(t_0\,,x))
	\qquad\text{a.s.}
\end{equation}
And when $t_0=0$, the same result can be shown to hold, using hands-on
methods, but with $(8/\pi)^{1/4}$ replaced by $(4/\pi)^{1/4}$.
To be concrete, we study solely the case $t_0=0$ here.
In that case, we of course have $u(t_0\,,x)=1$, which simplifies the exposition
somewhat.

For every $c>0$ consider the random set
\[
	\mathscr{U}(c) := \left\{x\in[0\,,1]:\ \limsup_{\varepsilon\downarrow0}
	\frac{u(\varepsilon\,,x)-1}{\varepsilon^{1/4}\sqrt{\log(1/\varepsilon)}}
	\ge c\right\}.
\]	
Our earlier remarks about \eqref{LIL} imply that $\mathscr{U}(c)$ has zero Lebesgue
measure a.s.\ for all $c>0$, and hence so does
\[
	\mathscr{U} := \bigcup_{c>0}\mathscr{U}(c).
\]
Among other things, the following theorem shows that
the Lebesgue-null set $\mathscr{U}$
has full Hausdorff dimension a.s.  

Throughout, $\dim_{_{\rm H}}$ and $\dim_{_{\rm P}}$ respectively
denote the Hausdorff and the packing dimension; see Matilla
\cite[Ch.s 4 and 5]{Matilla}.
We will also let $\dim_{_{\rm M}}$ denote the upper Minkowski---or box---dimension;
see the book by Matilla ({\it ibid.}).

\begin{theorem}\label{th:main}
	If $0<c\le \sigma(1)/\pi^{1/4}$, then with probability one,
		\[
			\dim_{_{\rm P}}(\mathscr{U}(c)) =1 \quad\text{and}\quad
			\dim_{_{\rm H}}(\mathscr{U}(c))= 1 - \frac{c^2\sqrt\pi}{|\sigma(1)|^2}.
		\]
	If $c>\sigma(1)/\pi^{1/4}$, 
	then $\mathscr{U}(c)=\varnothing$ a.s.
\end{theorem}


It has recently been shown by Kunwoo Kim, Yimin Xiao, and the second author
\cite{KKX1,KKX2}  that the 
largest global oscillations of the random
map $(t\,,x)\mapsto u(t\,,x)$ form an asymptotic multifractal; this property
had been predicted earlier in a voluminous physics literature on
the present ``intermittent'' stochastic systems [as well as for more
complex systems]. It is not hard to deduce
from Theorem \ref{th:main} that the high local oscillations of 
$(t\,,x)\mapsto u(t\,,x)$ are also multifractal
in a precise sense that we now explain next. Indeed,  let us observe that,
by Theorem \ref{th:main},
\begin{equation}\label{RHS}
	\mathscr{U} = \left\{ x\in[0\,,1]:\ 0<\limsup_{\varepsilon\downarrow0}
	\frac{u(\varepsilon\,,x)-1}{\varepsilon^{1/4}\sqrt{\log(1/\varepsilon)}}
	<\infty\right\}
	\qquad\text{a.s.}
\end{equation}
Moreover,  
the right-hand side of \eqref{RHS} 
is equal to $\cup_{c\in I} \mathscr{U}(c)$---where $I:=(0\,,\sigma(1)/\pi^{1/4}]$---%
and $c\mapsto\dim_{_{\rm H}}\mathscr{U}(c)$ is strictly decreasing on 
$I$. Thus, Theorem \ref{th:main} implies that the right-hand side of
\eqref{RHS} is a ``multifractal.''

Finally, let us recall that by a ``solution'' $u=u(t\,,x)$ to equation \eqref{SHE} we 
mean that $u$ is a ``mild'' solution to \eqref{SHE}. That is, $u$ 
satisfies the following for each $t>0$ and $x\in \R$: 
\begin{equation}
	u(t\,,x)=1+ \int_{(0,t)\times\R} p_{t-s}(x-y)\sigma(u(s\,,y))\,\xi(\d s\, \d y)
	\qquad\text{a.s.},
\end{equation}
where $p_r(w) := (2\pi r)^{-1/2}\e^{-w^2/(2r)}$
[$r>0$, $w\in\R$]
denotes the heat kernel on $\R$, and the stochastic integral is 
defined by Walsh \cite{Walsh}.

\section{The constant-coefficient case}

Consider first the constant-coefficient stochastic heat equation,
\begin{equation}\label{L-SHE}\left[\begin{split}
	&\dot{Z} = \tfrac12 Z'' + \xi\quad\text{ on $(0\,,\infty)\times\R$},\\
	&\text{subject to}\ Z(0)\equiv0.
\end{split}\right.\end{equation}
It is easy to see that $Z(t\,,x)=u(t\,,x)-1$,
where $u$ solves \eqref{SHE} in the special case
that $\sigma\equiv1$.

According to the theory of Walsh \cite[Ch.\ 3]{Walsh}, the solution to \eqref{L-SHE}
can be written, in mild form, as the following Wiener integral process:
\begin{equation}\label{Z}
	Z(t\,,x) = \int_{(0,t)\times\R} p_{t-s}(y-x)\,\xi(\d s\,\d y)\,.
\end{equation}

Evidently, $Z$ is a centered, Gaussian random field. The following simple
computation contains all of the
requisite information about the process $Z$.

\begin{lemma}\label{lem:3.3}
	For all $\varepsilon>0$ and $x,x'\in\R$,
	\[
		\Cov\left[ Z(\varepsilon\,,x)\,,
		 Z(\varepsilon\,,x') \right]
		= \frac{\sqrt\varepsilon}{2}g_2\left(\frac{x-x'}{\sqrt\varepsilon}\right),
	\]
	where $g$ denotes the \emph{incomplete Green's function} of the heat kernel;
	that is, $g_t(a) := \int_0^t p_r(a)\,\d r$ for all $t>0$ and $a\in\R$.
\end{lemma}

\begin{proof} 
	Let $\delta:=x-x'$. By the Wiener isometry and \eqref{Z},
	\[
		\Cov\left[ Z(\varepsilon\,,x)\,,
		 Z(\varepsilon\,,x') \right] = \int_0^\varepsilon\d s\int_{-\infty}^\infty\d y\
		 p_{\varepsilon-s}(y-x)p_{\varepsilon-s}(y-x')
		 =\int_0^\varepsilon p_{2s}(\delta)\,\d s.
	\]
	We have used the semigroup property of the heat kernel in the second identity.
	We can change variables $[r:=2s/\varepsilon]$ and observe that
	$p_{r\varepsilon}(\delta)=\varepsilon^{-1/2}p_r(\delta/\sqrt\varepsilon)$
	to finish.
\end{proof}

By l'H\^opital's rule, $g_2(a) \sim 4\pi^{-1/2}a^{-2}\exp ( -a^2/4)$ 
as $a\to\infty$. Therefore, Lemma \ref{lem:3.3},
and the strict positivity and the continuity of $g_2$ together
imply the following: The  Gaussian random field $Z(\varepsilon)$ is stationary
for every fixed $\varepsilon>0$. Furthermore, for every
real number $\alpha$,
\[
	\Cov\left[ Z(\varepsilon\,,0)\,,
	Z\left(\varepsilon\,,\varepsilon^{(1/2)+\alpha}\right) \right] 
	\asymp\begin{cases}
		\varepsilon^{1/2}&\text{if $\alpha\ge0$},\\
		 \varepsilon^{(1/2)+2|\alpha|}\,
		\exp\left(-\dfrac{1}{4\varepsilon^{2|\alpha|}}\right)&\text{if $\alpha<0$}, 
	\end{cases}
\]
uniformly for all $\varepsilon\in(0\,,1)$.
This implies roughly
that if $|x-y|\gg\sqrt\varepsilon$ then 
$Z(\varepsilon\,,x)$ and $Z(\varepsilon\,,y)$ are very close to being uncorrelated,
whereas $Z(\varepsilon)$ locally
behaves as a random constant on a spatial scale of $O(\sqrt\varepsilon)$. In other words, the ``correlation length'' of the random field $Z(\varepsilon)$
is $\sqrt\varepsilon$ when $\varepsilon\approx0$.

The rest of this section is devoted to developing some hitting estimates
for $Z$. The latter is a far simpler object than the solution $u$ to \eqref{SHE}
when $\sigma$ is a non-linear function. Therefore, one can hope for better information
about $Z$ than $u$.

\subsection{Upper bounds}

Because $Z$ is a Gaussian random field, we can appeal to concentration
of measure ideas in order to bound the local supremum of $Z$. The
following entropy estimate contains the key step in that direction.

\begin{lemma}\label{lem:E:max:max:Z}
	Choose and fix a real number $R>0$,
	and define
	\begin{equation}\label{B}
		\mathcal{B}_R(\varepsilon) := (0\,,\varepsilon]\times
		\left[0\,,R\sqrt\varepsilon\right]
		\qquad\text{for all $\varepsilon>0$}.
	\end{equation}
	Then, $\E[\sup_{\mathcal{B}_R(\varepsilon)}Z] \apprle\varepsilon^{1/4}$,
	uniformly for all $\varepsilon\in(0\,,1)$.
\end{lemma}

\begin{proof}
	It is well known that
	$\sqrt{\E ( |Z(t\,,x)-Z(s\,,y)|^2 )}
	\apprle d ( (t\,,x)\,;(s\,,y) ),$
	uniformly for all  $s,t\ge0$ and $x,y\in\R$, 
	where $d$ denotes the spatio-temporal
	distance function,
	\begin{equation}\label{d}
		d\left( (t\,,x)\,;(s\,,y)\right) := |t-s|^{1/4} + |x-y|^{1/2}.
	\end{equation}
	This is a well-known part of the folklore of the subject,
	and appears upon close inspection within
	the proof of Corollary 3.4 of Walsh \cite[p.\ 318]{Walsh},
	for example.
	The details are worked out pedagogically in \cite[(135), p.\ 31]{minicourse}.
	
	Next we note that by Dudley's theorem \cite{Dudley},
	\begin{equation}\label{Dudley}
		\E \sup_{\mathcal{B}_R(\varepsilon)}Z \apprle
		\int_0^{D(\varepsilon)}
		\sqrt{\log N_\varepsilon(r)}\,\d r,
	\end{equation}
	uniformly for all $\varepsilon\in(0\,,1)$, where:
	\begin{compactenum}[(a)]
		\item $D(\varepsilon)$ denotes the $d$-diameter of
			$\mathcal{B}_R(\varepsilon)$; 
		\item $N_\varepsilon$ is the metric entropy
			of $\mathcal{B}_R(\varepsilon)$; that is,
			$N_\varepsilon(r)$ denotes the minimum
			number of $d$-balls of radius $r>0$ that are needed to cover
			$\mathcal{B}_R(\varepsilon)$; and
		\item The implied constant in the inequality \eqref{Dudley} does not depend on
			$\varepsilon$ (consult, for example, Marcus and Rosen 
			\cite[Theorem 6.2, p.\ 245]{MR}).
	\end{compactenum}
	We may consider covering $\mathcal{B}_R(\varepsilon)$
	by $d$-balls of the form
	$\{(s\,,y): |s-t|^{1/4}+ |y-x|^{1/2}\leq r\}$ for every fixed $t,r>0$
	and $x\in\R$. Clearly, the number of such $d$-balls
	of radius $r>0$ that are needed to cover
	$\mathcal{B}_R(\varepsilon)$
	is of sharp order
	\[
		\left( 1+\frac{\varepsilon}{r^4}\right)
		\left( 1 + \frac{\sqrt\varepsilon}{r^2}\right) \asymp
		\left( 1 + \frac{\varepsilon}{r^4}\right)^{3/2} :=
		M_\varepsilon(r),
	\]
	uniformly for all $\varepsilon>0$ and $r\in(0\,,D(\varepsilon))$.
	It is not hard to see that $D(\varepsilon)\asymp\varepsilon^{1/4}$,
	uniformly for all $\varepsilon\in(0\,,1)$. Because
	$N_\varepsilon\apprle M_\varepsilon$
	uniformly on $(0\,,D(\varepsilon))$,
	it follows from \eqref{Dudley} that there exists
	a real number $C>0$ such that
	\[
		\E\sup_{\mathcal{B}_R(\varepsilon)}Z 
		\apprle\int_0^{C\varepsilon^{1/4}}
		\sqrt{1+\log\left( 1  +\frac{\varepsilon}{r^4}\right)}\,\d r
		\propto\varepsilon^{1/4},
	\]
	uniformly for all $\varepsilon\in(0\,,1)$. This completes
	the proof.
\end{proof}

Next we use Lemma \ref{lem:E:max:max:Z} and concentration of measure
in order to deduce an asymptotically-sharp maximal inequality for $Z$.

\begin{lemma}\label{lem:CoM}
	Choose and fix a real number $R>0$, and recall \eqref{B}.
	Then,
	\[
		\lim_{\lambda\to\infty}\frac{1}{\lambda^2}\log
		\P\left\{ \sup_{\mathcal{B}_R(\varepsilon)}Z \ge
		\left(\frac{4\varepsilon}{\pi}\right)^{1/4} \lambda\right\}=-1,
	\]
	uniformly for all $\varepsilon\in(0\,,1)$.
\end{lemma}

\begin{proof}
	In accord with Lemma \ref{lem:3.3},
	the Gaussian random field $Z(\varepsilon)$ is stationary and
	centered, with 
	$\Var[Z(\varepsilon\,,0)]=\sqrt{\varepsilon/\pi}$.
	Therefore, the inequality of Borell \cite{Borell}
	and Sudakov and T'sirelson \cite{ST}
	implies that $\sup_{\mathcal{B}_R(\varepsilon)}Z - 
	\E\sup_{\mathcal{B}_R(\varepsilon)}Z$ has sub-Gaussian tails;
	see Ledoux \cite[Ch.\ 7]{Ledoux}. 
	After a few lines of computations, this fact and Lemma \ref{lem:E:max:max:Z}
	together imply that
	\[
		\sup_{\varepsilon\in(0,1)}\log
		\P\left\{ \sup_{\mathcal{B}_R(\varepsilon)}Z \ge
		\left(\frac{4\varepsilon}{\pi}\right)^{1/4} \lambda\right\}
		\le -\lambda^2+o(\lambda),
	\]
	as $\lambda\to\infty$. At the same time,
	\[
		\P\left\{ \sup_{\mathcal{B}_R(\varepsilon)}Z \ge
		\left(\frac{4\varepsilon}{\pi}\right)^{1/4} \lambda\right\}
		\ge\P\left\{ Z(\varepsilon\,,0)\ge
		\left(\frac{4\varepsilon}{\pi}\right)^{1/4} \lambda\right\}
		= \P\{ X\ge 2^{1/2}\lambda\},
	\]
	where $X$ has a standard normal distribution; see Lemma \ref{lem:3.3}.
	Since $\log\P\{X\ge2^{1/2}\lambda\}\ge -\lambda^2+o(\lambda),$
	as $\lambda\to\infty$,
	the result follows.
\end{proof}

Define, for all $c>0$, the random set
\begin{equation}\label{G}
	\mathscr{G}(c) := 
	\left\{ y\in[0\,,1]:\ \limsup_{\varepsilon\downarrow0}
	\frac{Z(\varepsilon\,,y)}{\varepsilon^{1/4}\sqrt{\log(1/\varepsilon)}}\ge c\right\}.
\end{equation}

\begin{lemma}\label{lem:UB:0}
	Let $E$ be a measurable subset of $[0\,,1]$. Then,
	$\P\{\mathscr{G}(c)\cap E=\varnothing\}=1$ for every
	$c>  (1/\pi)^{1/4}\sqrt{\dim_{_{\rm M}}(E)}$.
\end{lemma}

\begin{proof}
	Define, for every $c>0$,
	\[
		\mathscr{L}(c) := \left\{ t\in[0\,,1]:\ 
		\sup_{y\in E}\ \frac{Z(t\,,y)}{t^{1/4}\sqrt{\log(1/t)}} \ge c \right\}.
	\]
	Every $\mathscr{L}(c)$ is a random subset of the time interval $[0\,,1]$.
	For all integers $n\ge 1$ define
	\[
		t_n := \e^{-\sqrt n},\quad\mathscr{T}(n) := \left[t_{n+1}\,,t_n\right],
		\quad\text{and}\quad
		\mathscr{S}(j;n) := \left[ j\sqrt{t_n}\,, (j+1)\sqrt{t_n}\right],
	\]
	for all integers $j\ge0$. Then, by monotonicity and a simple union bound,
	\begin{align*}
		\P\left\{ \mathscr{L}(c)\cap \mathscr{T}(n)\neq\varnothing\right\}
		&\le\sum_{\substack{%
			0\le j< 1/{t_n}:\\\mathscr{S}(j;n)\cap E\neq\varnothing}}
			\P\left\{ \sup_{t\in \mathscr{T}(n)}
			\sup_{y\in \mathscr{S}(j;n)} Z(t\,,y) > 
			c(nt_{n+1})^{1/4}\right\}.
	\end{align*}
	
	Recall the space-time sets $\mathcal{B}_R(\varepsilon)$ from
	\eqref{B}. Since $y\mapsto Z(t\,,y)$ is stationary, the natural logarithm of
	the $j$th term of the
	preceding sum is equal to 
	\[
		\log\P\left\{ \sup_{t\in \mathscr{T}(n)}
		\sup_{y\in \mathscr{S}(0;n)
		} Z(t\,,y) > c(nt_{n+1})^{1/4}\right\}
		\le \log\P\left\{ \sup_{\mathcal{B}_1(t_n)}Z
		> c(nt_{n+1})^{1/4}\right\}\sim -\frac{c^2\sqrt{\pi n}}{2} ,
	\]
	as $n\to\infty$,
	thanks to Lemma \ref{lem:CoM}
	and the fact that $t_{n+1}/t_n = 1 - (\frac12 + o(1))n^{-1/2}$
	as $n\to\infty$. Therefore, by the definition of the Minkowski
	dimension,
	\[
		\limsup_{n\to\infty}
		\frac{\log \P\left\{ \mathscr{L}(c)
		\cap \mathscr{T}(n)\neq\varnothing\right\}}{\sqrt n}
		\le - \frac{c^2\sqrt\pi}{2} + \frac{\dim_{_{\rm M}}(E)}{2}.
	\]
	In particular, if the right-hand side is $<0$, then $n\mapsto 
	\P\{\mathscr{L}(c)\cap \mathscr{T}(n)\neq\varnothing\}$
	sums to a finite number.
	Therefore, the Borel--Cantelli lemma implies that
	\[
		\limsup_{n\to\infty}\left( \mathscr{L}(c)
		\cap \mathscr{T}(n)\right)=\varnothing\text{ a.s.}\quad
		\forall c> (1/\pi)^{1/4}\sqrt{\dim_{_{\rm M}}(E)}.
	\]
	It is easy to see that 
	$\mathscr{G}(c)\cap E=\varnothing$ a.s.\ on
	the event that $\limsup_{n\to\infty}(\mathscr{L}(c)
	\cap \mathscr{T}(n))=\varnothing$, whence follows the lemma.
\end{proof}

It is easy to apply Lemma \ref{lem:UB:0} in order to improve itself
slightly. Note that the difference between the following and Lemma
\ref{lem:UB:0} is one about the Minkowski versus the packing dimension
of $E$.

\begin{lemma}\label{lem:UB:1}
	Let $E$ be a measurable subset of $[0\,,1]$. Then,
	$\P\{\mathscr{G}(c)\cap E=\varnothing\}=1$ for every
	$c>  (1/\pi)^{1/4}\sqrt{\dim_{_{\rm P}}(E)}.$
\end{lemma}

\begin{proof}
	Choose and fix $\eta>0$ small enough to ensure that
	$c>(1/\pi)^{1/4}\sqrt{\dim_{_{\rm P}}(E)+\eta}$.
	By the definition of packing dimension \cite[p.\ 81]{Matilla}
	there exists a closed cover $F_1,\ldots,F_m$ of $E$ 
	such that
	$\dim_{_{\rm P}}(E) \ge \max_{1\le j\le m}\dim_{_{\rm M}}(F_j) - \eta.$
	Thus, $\mathscr{G}(c)\cap F_j=\varnothing$ a.s.\ for all $1\le j\le m$ by Lemma \ref{lem:UB:0}.
	This completes the proof.
\end{proof} 

\subsection{Lower bounds}
As was mentioned earlier, the ``correlation length'' of the
spatial process $Z(\varepsilon)$ is of sharp order $\sqrt\varepsilon$
when $\varepsilon$ is small. It is possible to state
and prove a much stronger, quantitative version of this assertion. In order to describe
that, let us define for all $x\in\R$ and $\varepsilon,\delta\in(0\,,1)$,
\begin{equation}\label{Gamma}
	\Gamma(x\,;\varepsilon,\delta) := \left[x -\sqrt{2\varepsilon\log (1/\delta)}\,,
	x+\sqrt{2\varepsilon\log (1/\delta)}\right],
\end{equation}
and
\begin{equation}\label{Z:delta}
	Z_\delta(\varepsilon\,,x) := \int_{(0,\varepsilon)
	\times \Gamma(x;\varepsilon,\delta)}
	p_{\varepsilon-s}(y-x)\,\xi(\d s\,\d y).
\end{equation}

\begin{lemma}\label{lem:ZZ}
	$\E( |Z(\varepsilon\,,x) - Z_\delta(\varepsilon\,,x)|^2)
	<\delta^2\sqrt\varepsilon$
	for all $\varepsilon,\delta\in(0\,,1)$ and $x\in\R$.
\end{lemma}

\begin{remark}\label{rem:ind}
	Recall that if $\phi,\psi\in L^2(\R_+\times\R)$ are orthogonal
	in $L^2(\R_+\times\R)$ then the Wiener integrals $\int\phi\,\d\xi$ and 
	$\int\psi\,\d\xi$ are independent. This observation has the following
	by-product: If $x_1,\ldots,x_k\in\R$ satisfy
	\begin{equation}\label{gap}
		\min_{i\neq j}|x_i-x_j| > \sqrt{8\varepsilon\log (1/\delta)},
	\end{equation}
	then
	$\{Z_\delta(\varepsilon\,,x_i)\}_{i=1}^k$ are independent and identically
	distributed. 	It follows that the gap condition \eqref{gap} ensures that
	$\{Z(\varepsilon\,,x_i)\}_{i=1}^k$ is uniformly to within
	$O(\delta\varepsilon^{1/4})$ of an i.i.d.\ sequence.
\end{remark}

\begin{proof}[Proof of Lemma \ref{lem:ZZ}]
	By stationarity, we may---and will---consider only the case that $x=0$.
	According to the Wiener isometry, 
	$\E(| Z(\varepsilon\,,0) - Z_\delta(\varepsilon\,,0)|^2)$ is equal to
	\[
		\int_0^\varepsilon\d s\int_{y\in\R:\ 
		|y|>\sqrt{2\varepsilon\vert\log\delta\vert}}\d y\ 
		|p_s(y)|^2\le \int_0^\varepsilon
		\P\left\{ |X| > \sqrt{\frac{4\varepsilon
		\vert\log\delta\vert}{s}}
		\right\}\, \frac{\d s}{\sqrt{2\pi s}},
	\]
	where $X$ has the standard normal distribution. 
	Elementary manipulations
	and the well-known tail bound
	$\P\{|X|>\lambda\}\le \exp(-\lambda^2/2)$ together yield the lemma.
\end{proof}

\begin{lemma}\label{lem:Z-Z:delta}
	Let $F_1,F_2,\ldots$ be finite subsets of $\R$ that
	satisfy $\log|F_n|\le\kappa(1+o(1))n$ as $n\to\infty$, for some
	$0\le\kappa<\infty$.
	Then, uniformly for all $0<\delta<1$ and $0<\nu<\exp(-2\kappa)$,
	\[
		\limsup_{n\to\infty}\max_{x\in F_n}
		\frac{\left| Z(\nu^n,x)-Z_\delta(\nu^n,x)\right|}{%
		\nu^{n/4}\sqrt{\log(1/\nu^n)}} \le \delta
		\qquad\text{a.s.}
	\]
\end{lemma}

\begin{proof}
	
	Lemma \ref{lem:ZZ} ensures that
	\[
		\Var\left[Z(\nu^n,0) - Z_\delta(\nu^n,0)\right]
		\le \nu^{n/2}\delta^2\qquad \text{for all $n\ge1$}.
	\]
	Thanks to this and stationarity, a standard Gaussian tail bound yields
	\[
		\frac1n\log \P\left\{ \max_{x\in F_n}\left|
		Z(\nu^n,x) - Z_\delta(\nu^n,x)
		\right| > \delta\nu^{n/4}\sqrt{\log(1/\nu^n)}\right\}
		\le- \frac{1}{2}\log(1/\nu) +\kappa+o(1),
	\]
	as $n\to\infty$. The right-hand side of the preceding display is
	strictly negative for all $n$ sufficiently large. 
	Therefore, the Borel--Cantelli implies the result.
\end{proof}

Finally, we derive a matching ``converse'' to Lemma \ref{lem:UB:1}.
The proof of the following result
borrows heavily ideas from the theory of limsup random fractals 
\cite{KPX}, see also \cite{KS}.

\begin{lemma}\label{lem:LB}
	Recall \eqref{G}, and 
	let $E$ denote an arbitrary measurable subset of $[0\,,1]$. Then,
	$\P\{\mathscr{G}(c)\cap E \neq \varnothing\}=1$ for every
	strictly positive constant $c< \pi^{-1/4}\sqrt{\dim_{_{\rm P}}(E)}.$
\end{lemma}

\begin{proof}
	Choose and fix a number $\gamma\in( 0\,,\dim_{_{\rm P}}(E)).$
	A theorem of Joyce and Preiss \cite{JoycePreiss} implies that there exists
	a \emph{compact} set $F\subset E$ such that $\dim_{_{\rm P}}(F)>\gamma$. Define
	$E_\star :=\cap_{n=1}^\infty \cup (F\cap I)$, where the union is over all
	dyadic subintervals of $[0\,,1]$ that have length $2^{-n}$ and satisfy
	$\dim_{_{\rm M}}(F\cap I)\ge\gamma$.
	Clearly, $E_\star\subset E$ is compact. And the definition of packing dimension
	implies that 
	\begin{equation}\label{Baire}
		\dim_{_{\rm M}}(E_\star\cap I) \ge\gamma
		\quad\text{for all open intervals $I$ that intersect $E_\star$}.
	\end{equation}
	
	Let 
	$\delta \in (0, 1/(2\pi^{1/4}))$ and $\nu\in(0\,,1)$
	be fixed. Also, fix an arbitrary open interval $I$ that intersects $E_\star$ and 
	let $K$ denote the Kolmogorov capacity---or packing numbers---of 
	the compact set $E_\star\cap I$.
	That is, for every $\varepsilon>0$, $K(\varepsilon)$ denote 
	the maximum integer $k$ such that
	we can find $a_1,\ldots,a_k\in E_\star\cap I$ with the property that 
	$\min_{1\le i\ne j\leq k}|a_i-a_j|>\varepsilon$. Recall \cite[p.\ 78]{Matilla} that
	$\dim_{_{\rm M}}(E_\star\cap I) = \limsup_{\varepsilon\downarrow0} 
	\log K(\varepsilon)/\log(1/\varepsilon)$.
	Because of \eqref{Baire},  the preceding fact has the following consequence:
	There exists an infinite collection
	$\mathfrak{N}$ of positive integers, and a sequence 
	$x_{1,n},\ldots,x_{N(n),n}\in E_\star\cap I$---%
	for every $n\ge1$---with the following properties:
	\begin{compactenum}[(a)]
		\item $\min_{1\le i\neq j\le N(n)}|x_{i,n}-x_{j,n}|> \sqrt{8\nu^n\log(1/\delta)}$
			for all $n\ge1$; 
		\item There exists $M>0$
			such that $N(n) \le M[\nu^n\log(1/\delta)]^{-\gamma/2}$ for all $n\ge 1$; and
		\item $N(n) \ge  [8\nu^n\log(1/\delta)]^{-\gamma/2}$ for all $n\in\mathfrak{N}$.
	\end{compactenum}
	
	By Remark \ref{rem:ind}, and because of the stationarity of 
	$x\mapsto Z_\delta(t\,,x)$,
	the random variables $\{Z_\delta(\nu^n,x_{i,n})\}_{i=1}^{N(n)}$ are i.i.d..
	
	For every $c>0$ and $n\in\mathfrak{N}$, let 
	\[
		S_n(c) := \sum_{i=1}^{N(n)}\1\left\{ Z_\delta(\nu^n,x_{i,n}) >
		c \nu^{n/4}\sqrt{\log(1/\nu^n)}\right\},
	\]
	where $\1 A$ denotes the indicator function of the event $A$.
	Then, $S_n(c)$ has a Binomial distribution, and hence
	\begin{equation}\label{P(S)}
		\P\left\{ S_n(c)=0\right\} \le \e^{-\E[S_n(c)]},
	\end{equation}
	thanks to an elementary calculation. 
	Because of Lemma \ref{lem:ZZ}, we know that 
	for all $n\ge 1$,
	\[
		\Var[Z_\delta(\nu^n,0)] \ge 
		\left(\|Z(\nu^n,0)\|_2 - \frac{\delta\pi^{1/4}\nu^{n/4}}{\pi^{1/4}}\right)^2
		= \frac{\nu^{n/2}(1-\delta\pi^{1/4})^2}{\sqrt\pi}. 
	\]
	Therefore a Gaussian tail bound yields that, uniformly for all $n\in\mathfrak{N}$,
	\[
		\E[S_n(c)]  \ge \exp\left\{(1+o(1))
		\frac{n \log(1/\nu)}{2}
		\left[\gamma - \frac{c^2\pi^{1/2}}{\left(
		1-\delta\pi^{1/4}\right)^2}\right]\right\},
	\]
	as $n\to\infty$ in $\mathfrak{N}$. The preceding expression goes to infinity
	as $n\to\infty$ provided that 
	\begin{equation}\label{c:LB}
		c <\pi^{-1/4} \left( 1-\delta\pi^{1/4}\right)\sqrt{\gamma}.
	\end{equation}
	Because of \eqref{P(S)}, it follows that
	$\P\{S_n(c)=0\}\to0$
	as $n\to\infty$ in $\mathfrak{N}$ whenever $c>0$ satisfies \eqref{c:LB}.
	
	Consider the open random sets,
	\[
		U_{n,\nu}(c) := \left\{ x\in\R:\ 
		\frac{Z_\delta(\nu^n,x)}{\nu^{n/4}\sqrt{\log(1/\nu^n)}} > c\right\}
		\qquad[n=1,2,\ldots].
	\]
	We have shown that if $c$ satisfies \eqref{c:LB}, then 
	$\cap_{k=1}^\infty\bigcup_{n=k}^\infty[U_{n,\nu}(c)\cap 
	E_\star\cap I]\neq\varnothing$ a.s.
	for every open interval $I$ that intersects $E_\star$.
	By the Baire category theorem, the random set $\mathscr{U}(c)$
	intersects $E_\star$---hence also $E$---almost surely.
	In particular, we may apply Lemma \ref{lem:Z-Z:delta} with 
	$F_n :=\cup_{i=1}^{N(n)}\{x_{i,n}\}$
	to see that:
	\begin{compactenum}[(a)]
	\item $\kappa=\frac12\gamma\log(1/\nu)$; and
	\item $\P\{ \mathscr{G}(c-\delta)\cap E_\star\neq\varnothing\} =1$
		provided that $c$ satisfies \eqref{c:LB} and $\nu<\exp(-2\kappa)=\nu^{\gamma}$.
	\end{compactenum}
	The latter
	condition on $\nu$ holds automatically because 
	$\gamma<\dim_{_{\rm P}}(E)\le1.$
	Because $\delta$ can be otherwise
	as small as we wish, and since $\gamma<\dim_{_{\rm P}}(E)$ can be as large as we want,
	this proves the lemma.
\end{proof}

Finally, let us describe the critical case that is left open in Lemma \ref{lem:LB}.

\begin{lemma}\label{lem:LB:1}
	The conclusion of  Lemma \ref{lem:LB} continues to hold when
	$c=\pi^{-1/4}\sqrt{\dim_{_{\rm P}}(E)}.$
\end{lemma}

\begin{proof}
	Choose and fix a countable sequence of real numbers
	$0<c_1<c_2<\cdots$ that converge upward to $c=
	\pi^{-1/4}\sqrt{\dim_{_{\rm P}}(E)}.$
	The proof of Lemma \ref{lem:LB} showed---for the same compact set
	$E_\star\subset E$ as before---that
	$\P\{\mathscr{G}(c_n)\cap E_\star\cap I\neq\varnothing\}=1$ for all $n\ge 1$ and all open intervals
	$I$ that intersect $E_\star$. An application of the Baire category theorem reveals that
	$\mathscr{G}(c)=\cap_{n=1}^\infty \mathscr{G}(c_n)$ a.s.\ intersects $E_\star$ and hence $E$.
\end{proof}

\section{Proof of Theorem \ref{th:main}}

The proof of Theorem \ref{th:main} hinges on a localization result
which relates the small-time behavior of $u$ to that of $Z$.
First, let us recall a moment estimate of Khoshnevisan,
Swanson, Xiao, and Zhang. When \eqref{SHE} is replaced
by the stochastic heat equation on a torus, an almost-sure version of the following
was derived in the seminal work of Hairer \cite{Hairer}, and played
an important role in the ensuing deep theory of regularity structures of
Hairer \cite{Hairer2}. An almost-sure version of the following 
[for \eqref{SHE}, as is, on the real line] can likely be
deduced also from the work of Hairer and Labb\'e \cite{HairerLabbe}.

\begin{proposition}[\protect{\cite{KSXZ}}]\label{pr:KSXZ}
	For every $k\in[2\,,\infty)$ and $\vartheta\in(0\,,2/5)$,
	\[
		\sup_{t\in(0,\varepsilon)}\sup_{x\in\R}\E\left( \left| u(t\,,x) - 
		1 - \sigma(1)Z(t\,,x)\right|^k\right) = 
		o\left( \varepsilon^{k\vartheta}\right)
		\qquad\text{as $\varepsilon\downarrow0$.}
	\]
\end{proposition}

We now apply Proposition \ref{pr:KSXZ}---with $k=2$---in order to establish the following.

\begin{proposition}\label{pr:unif:moments}
	For every $\vartheta\in(0\,,2/5)$ and $N,M\ge0$,
	\[
		\sup_{t\in(0,\varepsilon)}\sup_{x\in[-M,N]}\left| u(t\,,x) - 1-
		\sigma(1)Z(t\,,x)\right| = o\left( \varepsilon^\vartheta\right)\quad\text{a.s.}
	\]
\end{proposition}

\begin{remark}
	If, instead of an SPDE
	on $\R_+\times\R$, we studied an analogous SPDE on $\R_+\times[0\,,1]$, then
	the work of Hairer and Pardoux \cite{HP} improves further the error rate 
	of the analogue of
	Proposition \ref{pr:unif:moments} to $O(\sqrt{\varepsilon|\log\varepsilon|})$. 
	We are not sure if the latter rate is the right one
	in the present setting, but  will not need these improvements and so will prove only what we need.
\end{remark}

\begin{proof}[Proof of Proposition \ref{pr:unif:moments}]
	Without loss of great generality, we study only the case that $M=0$ and $N=1$; the general case
	is proved similarly. 
	
	Let $\Delta(t\,,x) := u(t\,,x) -  1-\sigma(1)Z(t\,,x)$
	for all $t\ge0$ and $x\in\R$,
	and for every $\varepsilon\in(0\,,1)$ define
	\[
		\mathfrak{X}(\varepsilon):=
		\left\{ j\varepsilon^{2\vartheta}:\, 0\le j<\varepsilon^{-2\vartheta},\, j\in\Z\right\}
		\quad\text{and}\quad
		\mathfrak{T}(\varepsilon):=
		\left\{j\varepsilon^{4\vartheta}:\, 0\le j<\varepsilon^{1-4\vartheta},\,j\in\Z\right\}.
	\]
	We may apply Proposition \ref{pr:KSXZ} [with an arbitrary $k\ge2$]
	in order to see that, for all real numbers $b\in(0\,,\vartheta)$,
	\[
		\P\left\{ \max_{t\in\mathfrak{T}(\varepsilon)}
		\max_{x\in\mathfrak{X}(\varepsilon)}
		| \Delta(t\,,x) | >\varepsilon^b\right\}
		\le \varepsilon^{-kb}\sum_{t\in\mathfrak{T}(\varepsilon)}
		\sum_{x\in\mathfrak{X}(\varepsilon)}
		\| \Delta(t\,,x)\|_k^k
		\apprle \varepsilon^{(k-6)\vartheta- k b+1},
	\]
	uniformly for all $\varepsilon\in(0\,,1)$. Therefore,
	\[
		\P\left\{ \sup_{t\in(0,\varepsilon)}\sup_{x\in[0,1]}
		| \Delta(t\,,x) | > 3\varepsilon^b\right\}
		\apprle \varepsilon^{(k-6)\vartheta - kb+1} + 
		\mathcal{P}(u\,;\varepsilon) + \mathcal{P}(Z\,;\varepsilon),
	\]
	where
	\[
		\mathcal{P}(\Phi\,;\varepsilon) := 
		\P\left\{ \sup_{ \substack {t\in(0,\varepsilon) 
		\\ s\in\mathfrak{T}(\varepsilon)\,,  |s-t|\leq \varepsilon^{4\vartheta}}}
		\sup_{\substack {x\in[0,1]\\ y\in\mathfrak{X}(\varepsilon)\,, |y-x|\leq \varepsilon^{2\vartheta}}}
		|\Phi(t\,,x)-\Phi(s\,,y)| > K(\Phi)\varepsilon^b\right\}
	\]
	and the symbol $\Phi$ is in $\{u\,,Z\}$, $K(u)=1$, and $K(Z)=1/\sigma(1)$.
	Because $\sigma(1)\neq0$  throughout,
	$K(\Phi)$ is well defined and finite. Therefore, by the work of Walsh 
	\cite[Ch.\ 3]{Walsh}, 
	\[
		\left\| \sup_{ \substack {t\in(0,\varepsilon) \\ s\in\mathfrak{T}(\varepsilon)\,,  
		|s-t|\leq \varepsilon^{4\vartheta}}}
		\sup_{\substack {x\in[0,1]\\ y\in\mathfrak{X}(\varepsilon)\,, |y-x|\leq \varepsilon^{2\vartheta}}}
		\left| \Phi(t\,,x)-\Phi(s\,,y)\right| \right\|_k
		\le\varepsilon^{\vartheta+o(1)}\qquad\text{as $\varepsilon\downarrow0,$}
	\]
	for both possible choices of $\Phi\in\{u\,,Z\}$.
	As a result, it can be deduced from Chebyshev's inequality that
	$\mathcal{P}(\Phi\,;\varepsilon)\le \varepsilon^{k(\vartheta-b)+o(1)}=
	O(\varepsilon^{k(\vartheta-b)})$. Thus, we find that
	\[
		\P\left\{ \sup_{t\in(0,\varepsilon)}\sup_{x\in[0,1]}
		| \Delta(t\,,x) | > 3\varepsilon^b\right\}
		\apprle \varepsilon^{(k-6)\vartheta - kb}\quad\text{as $\varepsilon\downarrow0$}.
	\]
	Replace $\varepsilon$ by $2^{-n}$ to see from a monotonicity argument and
	the Borel--Cantelli lemma that, almost surely,
	$\sup_{t\in(0,\varepsilon)}\sup_{x\in[0,1]}
	|\Delta(t\,,x) |=O(\varepsilon^b)$ as $\varepsilon\downarrow0$.
	The result follows from this because $k\ge2$,
	$\vartheta\in(0\,,2/5)$, $b\in(0\,,(k-6)\vartheta/k)$
	are all otherwise arbitrary.
\end{proof}

We conclude the paper with the following.

\begin{proof}[Proof of Theorem \ref{th:main}]
	We will assume throughout that $E$ is a subset of $[0\,,1]$. A routine
	change of scale [in the arguments] will yield the general case.
	
	Proposition \ref{pr:unif:moments} implies that, for every real number $c>0$,
	$\mathscr{U}(c) = \mathscr{G}(c/\sigma(1))$ a.s.,
	where the random sets $\mathscr{G}(\bullet)$ were defined in \eqref{G}.
	It follows from this and Lemmas \ref{lem:UB:1}, \ref{lem:LB},
	and \ref{lem:LB:1} that
	for all measurable sets $E\subset[0\,,1]$:
	\begin{compactenum}
		\item[\bf A1.] If $c\le \pi^{-1/4}\sigma(1)\sqrt{\dim_{_{\rm P}}(E)}$, 
			then $\P\{\mathscr{U}(c)\cap E\neq\varnothing\}=1$; and
		\item[\bf A2.] If $c>(1/\pi)^{1/4}\sigma(1)\sqrt{\dim_{_{\rm P}}(E)}$,
			then $\P\{\mathscr{U}(c)\cap E\neq\varnothing\}=0$; and
	\end{compactenum}
	
	According to the theory of limsup random fractals \cite{KPX}, 
	for every $\rho\in(0\,,1)$
	there exists a random set $\Sigma_\rho\subset[0\,,1]$, that is 
	independent of the process $u$, and
	satisfies the following for all measurable sets $F\subset[0\,,1]$:
	\begin{compactenum}
		\item[\bf B1.] $\P\{\Sigma_\rho\cap F\neq\varnothing\}=1$ 
			if $\rho<\dim_{_{\rm P}}(F)$. In this case,
			$\dim_{_{\rm P}}(\Sigma_\rho\cap F)=\dim_{_{\rm P}}(F)$; and
		\item[\bf B2.] $\P\{\Sigma_\rho\cap F\neq\varnothing\}=0$ 
		if $\rho>\dim_{_{\rm P}}(F)$.
	\end{compactenum}
	Both {\bf B1} and {\bf B2} follow from Theorems 3.1 and 3.2 of \cite{KPX}.
	
	By the independence of
	$\Sigma_\rho$ and $\mathscr{U}(c)$, we may combine 
	{\bf A1} and {\bf B1}--{\bf B2}
	with $F=\mathscr{U}(c)$ and $E=\Sigma_\rho$ in order to see that
	$\dim_{_{\rm P}}(\Sigma_\rho)=1$ and:
	\[
		c\le \sigma(1)/\pi^{1/4}\ \Rightarrow\ 
		\P\{\mathscr{U}(c)\cap\Sigma_\rho\neq\varnothing\}=1\
		\Rightarrow\ \rho\le \dim_{_{\rm P}}(\mathscr{U}(c)).
	\]
	Since $\rho\in(0\,,1)$ was arbitrary, it follows that
	$\dim_{_{\rm P}}(\mathscr{U}(c))=1$ a.s.\ 
	for all $c\le \sigma(1)/\pi^{1/4}$.
	On the other hand, {\bf A2} shows that if $c>\sigma(1)/\pi^{1/4}$, then
	$\mathscr{U}(c)=\varnothing$ a.s.
	
	Finally, the computation of the Hausdorff dimension of $\mathscr{U}(c)$ follows from 
	{\bf A1} and {\bf A2}, using  a codimension argument which we skip;
	see Peres \cite{Peres}
	and Khoshnevisan \cite[\S4.7, p.\ 435]{MPP} for details.
\end{proof}

\section*{Acknowledgement} We are grateful to an anonymous 
referee for his or her careful reading of the paper, and for
making numerous constructive comments that improved the 
presentation, and well as content, of this paper.

\spacing{.95}

\bigskip
\small
\noindent\textbf{Jingyu Huang} [\texttt{jhuang@math.utah.edu}]\\
\noindent 
	Department of Mathematics, University of Utah, Salt Lake City ,UT 84112-0090\\
	
\noindent\textbf{Davar Khoshnevisan} [\texttt{davar@math.utah.edu}]\\
\noindent	Department of Mathematics, University of Utah, Salt Lake City ,UT 84112-0090\\[.2cm]

\end{document}